\documentclass[12pt,a4paper]{article}
\usepackage{cmap}
\usepackage[T1]{fontenc}
\usepackage{lmodern}
\usepackage{silence}
\usepackage{microtype}
\usepackage[english]{babel}
\usepackage{geometry}
\usepackage{amsmath,amsfonts,amssymb,mathtools}
\usepackage{enumitem}
\usepackage{amsthm}
\usepackage{hyperref}
\hypersetup{
	colorlinks=true,
	linkcolor=blue!70!black,
	citecolor=green!50!black,
	urlcolor=blue!60!black,
	pdftitle={Nullcline geometry constrains the location of oscillatory instabilities in planar predator--prey systems},
	pdfauthor={E. Chan-Lopez, A. Martin-Ruiz, V. Castellanos},
	pdfsubject={Nonlinear dynamics and bifurcation theory in predator-prey systems},
	pdfkeywords={Hopf bifurcation, Neimark-Sacker bifurcation, predator-prey systems, nullcline geometry, nonlinear dynamics}
}
\usepackage[capitalize,nameinlink]{cleveref}
\usepackage{aliascnt}
\usepackage{tikz}
\usetikzlibrary{calc,math,positioning, arrows.meta, shapes.geometric, decorations.pathreplacing}

\newtheorem{theorem}{Theorem}

\newtheorem{lemma}[theorem]{Lemma}
\newtheorem{corollary}[theorem]{Corollary}
\theoremstyle{definition}
\newtheorem{definition}[theorem]{Definition}
\newtheorem{remark}[theorem]{Remark}
\newcommand{\R}{\mathbb{R}}
\newcommand{\tr}{\operatorname{tr}}
\allowdisplaybreaks

\title{\textbf{Nullcline geometry constrains the location\\
		of oscillatory instabilities in planar predator--prey systems}}
\author{E.\ Chan-L\'opez\textsuperscript{1}\thanks{Corresponding author:
		\texttt{eduardo.clopez13@gmail.com}}
	\and A.\ Mart\'in-Ruiz\textsuperscript{2}
	\and V.\ Castellanos\textsuperscript{1}}
\date{%
	\textsuperscript{1}Divisi\'on Acad\'emica de Ciencias B\'asicas,
	Universidad Ju\'arez Aut\'onoma de Tabasco, 86690 Cunduac\'an,
	Tabasco, Mexico\\
	\textsuperscript{2}Instituto de Ciencias Nucleares, Universidad Nacional
	Aut\'onoma de M\'exico, 04510 Ciudad de M\'exico, Mexico}

\begin{document}
	\maketitle
	
	\begin{abstract}
		We prove that the critical structure of the prey nullcline governs where
		oscillatory instabilities can emerge in planar predator--prey systems.
		For a broad class of Gause--type models, we establish a general geometric
		localization theorem: the prey coordinate of every Hopf bifurcation point
		is confined between consecutive critical points of the prey nullcline.
		The mechanism is purely geometric. Along the nullcline, the diagonal
		Jacobian entry $J_{11}$ is proportional to the nullcline slope $g'(x)$,
		independently of the bifurcation parameter, while $J_{22}\le 0$ at any
		coexistence equilibrium. At critical points of the nullcline
		($g'=0$), the Jacobian trace becomes strictly negative, creating a
		spectral barrier that precludes oscillatory instability. This barrier
		partitions the state space into dynamically distinct regions, confining
		the onset of limit--cycle oscillations to the ascending branches of the
		nullcline. We illustrate the principle in three canonical families---
		Bazykin's model (quadratic nullcline), a Leslie--Holling type~IV system
		with harvesting (cubic nullcline), and the Crowley--Martin model with
		predator interference (rational nullcline)---for which we obtain
		closed--form Hopf bifurcation loci. The same geometric mechanism extends
		to discrete time: for the forward Euler map with step size $\tau$, an
		exact identity $\det(J^G) - 1 = \tau[\tr(J)+\tau\det(J)]$ shows that
		the Neimark--Sacker locus is strictly disjoint from the continuous
		Hopf locus yet tracks it on the same ascending branch at an
		$O(\tau)$ distance, crossing the spectral barrier only in the
		coarse--step regime. The results demonstrate that the functional responses
		responsible for predator saturation and interference not only bound
		the efficiency of interactions but also imprint a geometric
		architecture on the bifurcation landscape.
	\end{abstract}
	
	
	
	\section{Introduction}\label{sec:intro}
	The onset of sustained oscillations in predator--prey systems is among
	the most fundamental phenomena in nonlinear ecology. The classical
	explanation, due to Rosenzweig and MacArthur~\cite{Rosenzweig1963},
	links this onset to a geometric feature of the prey nullcline: a
	coexistence equilibrium destabilizes via a Hopf bifurcation when it
	crosses the vertex of the hump--shaped nullcline, giving rise to
	limit--cycle dynamics and the paradox of
	enrichment~\cite{Rosenzweig1971}. This graphical criterion has remained
	one of the most productive organizing principles in mathematical
	ecology, because it connects a static geometric property---the shape
	of the nullcline---to a dynamic transition---the emergence of
	oscillations.
	
	In models that incorporate richer ecological mechanisms---intraspecific
	predator competition~\cite{Bazykin1998}, non--monotone functional
	responses that capture group defense~\cite{Andrews1968,LiXiao2007,
		HuangXiaZhang2016}, harvesting~\cite{ChengZhang2021}, or Allee
	effects~\cite{ShangQiao2023,ZhangEtAl2025}---the bifurcation landscape
	becomes considerably more complex. Multiple coexistence equilibria
	appear and disappear through saddle--node bifurcations; Bogdanov--Takens
	points organize homoclinic structures and limit--cycle creation;
	degenerate Hopf bifurcations of codimension up to five govern the number
	and stability of coexisting limit cycles. In this richer setting, a
	natural question arises: \emph{does the nullcline geometry continue to
		organize the bifurcation structure, and if so, how?}
	
	The purpose of this paper is to show that a precise geometric
	localization principle governs the Hopf bifurcation in all these
	enriched models. We prove, for a large class of Gause--type predator--prey
	systems, that the prey coordinate of any Hopf bifurcation point is
	confined between consecutive critical points of the prey nullcline. The
	confinement mechanism is the same in every case: the critical points of
	the nullcline impose an algebraic constraint on the Jacobian that
	prevents the spectral condition for oscillatory instability from being
	satisfied.
	
	More precisely, at any coexistence equilibrium on the prey nullcline
	\(y = g(x)\), the \((1,1)\)-entry of the Jacobian decomposes as
	\begin{equation}\label{eq:intro_decomp}
		\tr(J) = \underbrace{J_{11}(x)}_{\text{sign governed by } g'(x)}
		+ \underbrace{J_{22}(x,y)}_{\leq\, 0}.
	\end{equation}
	At a critical point \(x_c\) where \(g'(x_c) = 0\), we obtain \(J_{11} = 0\),
	whence \(\tr(J) = J_{22} < 0\): the equilibrium is necessarily stable,
	and no oscillatory instability can develop. On the descending branch
	(\(g' < 0\)), \(J_{11} < 0\) reinforces this stability. Only on ascending
	branches (\(g' > 0\)) can \(J_{11}\) contribute a positive term large enough
	to drive the trace through zero---the threshold for Hopf bifurcation.
	The critical points of the nullcline thus act as \emph{spectral
		barriers} that partition the state space into regions of qualitatively
	distinct dynamical behavior.
	
	This geometric principle has several consequences:
	\begin{enumerate}[label=(\roman*)]
		\item A general theorem that captures the mechanism under mild
		hypotheses on the functional response and predator self--regulation.
		\item Explicit localization formulas for three model families
		spanning quadratic, cubic, and rational nullcline geometries.
		\item For the Crowley--Martin model with predator interference, a
		closed--form expression for the Hopf bifurcation locus that recovers
		the Rosenzweig--MacArthur criterion in the limit of vanishing
		interference.
		\item An extension to discrete--time systems: for the forward Euler
		discretization with step size $\tau$, an exact identity
		$\det(J^G) - 1 = \tau[\tr(J) + \tau\det(J)]$ shows that the
		Neimark--Sacker locus is strictly disjoint from the continuous Hopf
		locus yet tracks it on the same ascending branch at an $O(\tau)$
		distance, crossing the spectral barrier onto the descending branch
		only in the coarse--step regime $\tau > |J_{22}|/\det(J)$.
		\item Because the Hopf bifurcation is the primary route to
		oscillatory dynamics, and the Bogdanov--Takens point organizes the
		nearby homoclinic and limit--cycle structure, the localization
		constrains the entire local bifurcation architecture to a
		geometrically determined region of state space.
	\end{enumerate}
	
	The geometric organization of the localization principle is summarized
	in \Cref{fig:architecture}. The prey nullcline vertex is a spectral
	barrier: at the vertex the prey contribution to the trace vanishes and
	no oscillatory onset is possible. The ascending branch is the admissible
	region for the continuous Hopf bifurcation and, for small Euler step,
	for the discrete Neimark--Sacker bifurcation as well; the two loci sit
	adjacent to one another, separated by a margin of order $\tau$. The
	spectral insets illustrate the corresponding eigenvalue transitions.
	This figure serves as a conceptual guide for the results developed in
	the subsequent sections.
	
	\begin{figure}[htbp]
		\centering
		\begin{tikzpicture}[>=Stealth, line cap=round, line join=round]
			\colorlet{hopfbg}{blue!8}
			\colorlet{hopfline}{blue!60!black}
			\colorlet{nsbg}{green!10}
			\colorlet{nsline}{green!50!black}
			\colorlet{curvecolor}{black!75}
			\colorlet{axiscolor}{black!60}
			
			\begin{scope}[shift={(0,0)}]
				\fill[hopfbg] (0.5, 0) -- (0.5, 1.46) .. controls (2.5, 4.5) and (4, 4.8) .. (5, 4.8) -- (5, 0) -- cycle;
				\fill[nsbg] (5, 0) -- (5, 4.8) .. controls (6, 4.8) and (7.5, 4.5) .. (9.5, 1.46) -- (9.5, 0) -- cycle;
				\draw[ultra thick, curvecolor] (0.5, 1.46) .. controls (2.5, 4.5) and (4, 4.8) .. (5, 4.8) .. controls (6, 4.8) and (7.5, 4.5) .. (9.5, 1.46);
				\draw[->, thick, axiscolor] (0,0) -- (10.5,0) node[right, text=black] {Prey density, $x$};
				\draw[->, thick, axiscolor] (0,0) -- (0,5.5) node[above, text=black] {Predator density, $y$};
				\draw[dashed, thick, black!50] (5, 4.8) -- (5, 0) node[below, text=black] {$x_v$};
				\fill (5, 4.8) circle (2.5pt) node[above=4pt] {$g'(x_v)=0$};
				\node[fill=white, inner sep=3pt, font=\small, text=black!80] at (5, 3.4) {Spectral boundary};
				\node[font=\normalsize, text=hopfline, align=center] at (2.7, 1.7) {Admissible Hopf /};
				\node[font=\normalsize, text=nsline, align=center] at (2.7, 1.3) {NS region ($g'>0$)};
				\node[font=\footnotesize, text=black!55, align=center] at (7.4, 1.5) {Stable region};
				\node[font=\footnotesize, text=black!55, align=center] at (7.4, 1.1) {(no oscillatory onset)};
			\end{scope}
			
			\begin{scope}[shift={(2.7, 8)}]
				\node[font=\small\bfseries] at (0, 2.5) {Continuous--time limit};
				\node[font=\footnotesize, text=black!70] at (0, 2.1) {Imaginary axis crossing};
				\draw[->, axiscolor] (-1.8,0) -- (1.8,0) node[right, font=\scriptsize, text=black] {Re};
				\draw[->, axiscolor] (0,-1.5) -- (0,1.5) node[above, font=\scriptsize, text=black] {Im};
				\draw[->, thick, hopfline] (-0.8, 0.8) -- (0.8, 0.8);
				\fill[white, draw=hopfline, thick] (-0.5, 0.8) circle (1.8pt); 
				\fill[hopfline] (0, 0.8) circle (2.5pt); 
				\draw[->, thick, hopfline] (-0.8, -0.8) -- (0.8, -0.8);
				\fill[white, draw=hopfline, thick] (-0.5, -0.8) circle (1.8pt);
				\fill[hopfline] (0, -0.8) circle (2.5pt);
			\end{scope}
			
			\begin{scope}[shift={(7.3, 8)}]
				\node[font=\small\bfseries] at (0, 2.45) {Discrete--time map};
				\node[font=\footnotesize, text=black!70] at (0, 2.1) {Unit circle crossing};
				\draw[->, axiscolor] (-1.8,0) -- (1.8,0) node[right, font=\scriptsize, text=black] {Re};
				\draw[->, axiscolor] (0,-1.5) -- (0,1.5) node[above, font=\scriptsize, text=black] {Im};
				\draw[thick, black!50] (0,0) circle (1.1);
				\draw[->, thick, nsline] (0.3, 0.3) -- (1.4, 1.4);
				\fill[white, draw=nsline, thick] (0.49, 0.49) circle (1.8pt); 
				\fill[nsline] (0.778, 0.778) circle (2.5pt); 
				\draw[->, thick, nsline] (0.3, -0.3) -- (1.4, -1.4);
				\fill[white, draw=nsline, thick] (0.49, -0.49) circle (1.8pt);
				\fill[nsline] (0.778, -0.778) circle (2.5pt);
			\end{scope}
			
			\draw[->, dashed, thick, hopfline] (2.7, 6.2) -- (2.7, 4.5);
			\draw[->, dashed, thick, nsline] (7.3, 6.2) -- (7.3, 4.5);
			
			\begin{scope}[shift={(10.0, 4.8)}]
				\node[right, font=\normalsize\bfseries] at (0, 3.8) {Spectral Rules};
				\draw[thick, black!20] (0, 3.4) -- (4.7, 3.4);
				\fill[hopfbg] (0, 1.5) rectangle (4.7, 2.9);
				\draw[left color=hopfline, right color=hopfbg, draw=none] (0, 1.5) rectangle (0.1, 2.9);
				\node[right, font=\small\bfseries, text=hopfline] at (0.2, 2.55) {Hopf Bifurcation};
				\node[right, font=\footnotesize] at (0.2, 2.15) {Critical trace: $\operatorname{tr}(J) = 0$};
				\node[right, font=\footnotesize] at (0.2, 1.75) {Constraint: $J_{11} \propto g'(x) > 0$};
				\fill[nsbg] (0, -0.5) rectangle (4.7, 0.9);
				\draw[left color=nsline, right color=nsbg, draw=none] (0, -0.5) rectangle (0.1, 0.9);
				\node[right, font=\small\bfseries, text=nsline] at (0.2, 0.55) {Neimark-Sacker};
				\node[right, font=\footnotesize] at (0.2, 0.15) {Critical det.: $\operatorname{det}(J_G) = 1$};
				\node[right, font=\footnotesize] at (0.2, -0.25) {Locus: $\operatorname{tr}(J) = -\tau\operatorname{det}(J)$};
				\draw[thick, black!20] (0, -1.0) -- (4.7, -1.0);
			\end{scope}
		\end{tikzpicture}
		\caption{Geometric organization of oscillatory instabilities along the
			prey nullcline. The critical point $x_{\mathrm{v}}$ acts as a spectral
			barrier: at the vertex $J_{11}=p(x)g'(x)=0$, so the trace reduces to
			$\tr(J)=J_{22}<0$ and no oscillatory onset is possible. The ascending
			branch ($g'>0$) is the admissible region for both the continuous Hopf
			bifurcation (where $\tr(J)=0$; eigenvalues cross the imaginary axis,
			upper--left inset) and, for small Euler step $\tau$, the discrete
			Neimark--Sacker bifurcation (where $\tr(J)=-\tau\det(J)$; eigenvalues
			cross the unit circle, upper-right inset). The two loci are distinct
			but adjacent, separated by a margin of order $\tau$ in accordance with
			the exact identity $\det(J^G)-1=\tau[\tr(J)+\tau\det(J)]$. The
			descending branch supports no oscillatory onset in the
			faithful--discretization regime. The localization is governed by the
			trace decomposition $\operatorname{tr}(J) = p(x)\,g'(x) + J_{22}$.}
		\label{fig:architecture}
	\end{figure}
	
	\subsection{Related work}\label{subsec:related}
	Hammoum, Sari, and Yadi~\cite{HammoumSariYadi2023} extended the
	Rosenzweig--MacArthur stability criterion to a general Gause model with
	variable predator mortality, defining an arc~$\mathcal{A}$ of the
	ascending branch where $\tr(J) \geq 0$ and computing first Lyapunov
	coefficients for several models. Their framework determines stability
	and criticality but does not produce closed--form Hopf loci, does not
	establish that $\mathcal{A}$ is strictly confined below the nullcline
	vertex, and requires monotone functional response ($p'(x) > 0$),
	excluding models with group defense.
	
	Lu and Huang~\cite{LuHuang2021} analyzed Bazykin's model in depth,
	uncovering degenerate Hopf and focus--type BT bifurcations of
	codimension~3. For the Holling type~IV Leslie system, Li and
	Xiao~\cite{LiXiao2007}, Huang et~al.~\cite{HuangXiaZhang2016}, Dai
	and Zhao~\cite{DaiZhao2018}, and Cheng and
	Zhang~\cite{ChengZhang2021} performed increasingly refined bifurcation
	analyses. These works address specific bifurcation phenomena but do not
	investigate the geometric constraints imposed by the nullcline on the
	admissible bifurcation regions.
	
	\section{Framework and the general localization theorem}\label{sec:general}
	We consider planar predator--prey systems of Gause type
	\begin{equation}\label{eq:gause}
		\dot{x} = x\,\phi(x) - y\,p(x),\qquad
		\dot{y} = y\,Q(x,y),
	\end{equation}
	defined on $\R^2_+$. The functions $\phi$, $p$ and $Q$ are assumed
	sufficiently smooth. A coexistence equilibrium (CEP) is a point
	$P^*=(x^*,y^*)$ with $x^*,y^*>0$ satisfying
	\begin{equation}\label{eq:gause_eq}
		x^*\phi(x^*) = y^*\,p(x^*),\qquad Q(x^*,y^*) = 0.
	\end{equation}
	The prey nullcline is obtained from the first equation:
	\[
	y = g(x) \coloneqq \frac{x\,\phi(x)}{p(x)},\qquad x>0.
	\]
	A Hopf bifurcation at a CEP requires that the Jacobian
	\[
	J = \begin{pmatrix}
		J_{11} & J_{12}\\
		J_{21} & J_{22}
	\end{pmatrix}
	\]
	satisfies $\tr(J)=0$ and $\det(J)>0$. The key observation is that on the
	nullcline the entry $J_{11}$ is governed entirely by the slope of $g$.
	
	\begin{lemma}\label{lem:J11_prop}
		For system~\eqref{eq:gause}, assume $p(x)>0$ for $x>0$. Then along the
		prey nullcline $y=g(x)$,
		\begin{equation}\label{eq:J11_gprime}
			J_{11}(x) = p(x)\,g'(x).
		\end{equation}
	\end{lemma}
	\begin{proof}
		Direct differentiation gives
		\[
		J_{11} = \frac{\partial}{\partial x}\bigl[x\phi(x)-y\,p(x)\bigr]
		= \phi(x)+x\phi'(x)-y\,p'(x).
		\]
		Evaluating on $y=g(x)=x\phi(x)/p(x)$,
		\[
		J_{11} = \phi(x)+x\phi'(x) - \frac{x\phi(x)}{p(x)}p'(x).
		\]
		On the other hand,
		\[
		g'(x) = \frac{[\phi(x)+x\phi'(x)]p(x) - x\phi(x)p'(x)}{p(x)^2},
		\]
		hence $p(x)g'(x) = \phi(x)+x\phi'(x) - \frac{x\phi(x)p'(x)}{p(x)} = J_{11}$.
	\end{proof}
	
	The sign of $J_{11}$ is therefore exactly the sign of $g'(x)$.
	
	\begin{remark}[Extension to predator--dependent functional responses]\label{rem:extended_lemma}
		Some standard models, such as Beddington--DeAngelis, Crowley--Martin,
		and Hassell--Varley, replace $p(x)$ by a predator-dependent functional
		response $\Phi(x,y)$, so the prey equation reads
		$\dot{x} = x\phi(x) - y\,\Phi(x,y)$. The nullcline is then defined
		implicitly by $g(x) = x\phi(x)/\Phi(x,g(x))$. Repeating the
		calculation of Lemma~\ref{lem:J11_prop} in this generalized setting
		yields
		\[
		J_{11}\big|_{y=g(x)} = \bigl[\Phi(x,g(x)) + g(x)\,\Phi_y(x,g(x))\bigr]\,g'(x).
		\]
		Whenever the bracket is positive, the sign of $J_{11}$ still equals
		the sign of $g'(x)$, and Theorem~\ref{thm:general} below applies
		verbatim. For the Crowley--Martin response
		$\Phi = ax/[(1+bx)(1+cy)]$ one finds
		$\Phi + g\Phi_y = \Phi/(1+cg) > 0$, so the bracket is positive on
		the entire nullcline.
	\end{remark}
	
	The predator self--regulation is encoded in $J_{22}$. At a CEP,
	$Q(x^*,y^*)=0$, so
	\begin{equation}\label{eq:J22}
		J_{22} = \frac{\partial}{\partial y}\bigl[y\,Q(x,y)\bigr]\Big|_{P^*}
		= Q(x^*,y^*) + y^*\,\frac{\partial Q}{\partial y}(x^*,y^*)
		= y^*\,\frac{\partial Q}{\partial y}(x^*,y^*).
	\end{equation}
	In all ecologically reasonable models, the predator growth rate $Q$
	decreases with predator density due to intraspecific competition,
	interference, or territoriality. We therefore impose the following
	standing hypothesis.
	
	\begin{definition}
		System~\eqref{eq:gause} is said to satisfy the \emph{predator
			self--damping hypothesis} if
		\[
		\frac{\partial Q}{\partial y}(x,y) \le 0 \quad \text{for all } x>0,\; y>0,
		\]
		and the inequality is strict at every coexistence equilibrium $P^*$.
	\end{definition}
	
	Under this hypothesis, $J_{22} < 0$ at any CEP. We now obtain the
	central spectral decomposition:
	\begin{equation}\label{eq:trace_decomp}
		\boxed{\ \tr(J) = p(x^*)\,g'(x^*) + J_{22}(x^*,y^*)\ },
	\end{equation}
	with $J_{22}<0$. The consequences are immediate.
	
	\begin{theorem}[Geometric localization of Hopf bifurcations]\label{thm:general}
		Let system~\eqref{eq:gause} satisfy $p(x)>0$ on $(0,\infty)$ and the
		predator self--damping hypothesis. If $P^*$ is a CEP where the system
		undergoes a Hopf bifurcation, then the prey nullcline must be strictly
		increasing at $P^*$:
		\[
		g'(x^*) > 0.
		\]
		Consequently, $x^*$ lies in an open interval between two consecutive
		critical points of $g$, on a branch where $g$ is ascending.
	\end{theorem}
	\begin{proof}
		From~\eqref{eq:trace_decomp}, $\tr(J)=0$ implies
		$p(x^*)g'(x^*) = -J_{22} > 0$. Since $p(x^*)>0$, we obtain
		$g'(x^*)>0$. The second statement follows because the sign of $g'$
		can change only at critical points.
	\end{proof}
	
	\begin{remark}
		The proof does not require specifying a bifurcation parameter; it holds
		for any parameter values for which a CEP exists and the Hopf condition
		is met. The localization is therefore structurally robust.
	\end{remark}
	
	\begin{corollary}[Bogdanov--Takens localization]\label{cor:BT}
		Under the same hypotheses, every Bogdanov--Takens point (where
		$\tr(J)=\det(J)=0$) is also confined to an ascending branch of the
		prey nullcline.
	\end{corollary}
	\begin{proof}
		The BT condition includes $\tr(J)=0$, so the argument of
		Theorem~\ref{thm:general} applies identically.
	\end{proof}
	
	\begin{remark}
		Since the BT point organizes the nearby homoclinic bifurcation curve and
		the fold of limit cycles, the geometric localization constrains not only
		the onset of oscillations but also the more complex codimension--two
		structures that emanate from it.
	\end{remark}
	
	The above theorem unifies and clarifies the three families studied
	below. In each case the specific functional responses determine the
	shape of $g$ and the resulting Hopf loci can be written in closed form,
	but the confinement of oscillatory instabilities to ascending branches
	follows from the general principle.
	
	\section{Quadratic nullcline: Bazykin model}\label{sec:quadratic}
	\subsection{Model and nullcline geometry}
	The Bazykin predator--prey model~\cite{Bazykin1998} with intraspecific
	predator competition is
	\begin{equation}\label{eq:bazykin}
		\begin{aligned}
			\dot{x} &= r\,x\!\left(1-\frac{x}{k}\right) - \frac{a\,x\,y}{x+b},\\[4pt]
			\dot{y} &= e\,\frac{a\,x\,y}{x+b} - d\,y - \sigma\,y^2.
		\end{aligned}
	\end{equation}
	This corresponds to $p(x) = ax/(x+b)$ and
	$Q(x,y) = e a x/(x+b) - d - \sigma y$, which satisfies the self--damping
	hypothesis because $\partial Q/\partial y = -\sigma < 0$.
	
	The prey nullcline is the parabola
	\[
	g(x) = \frac{r(k-x)(b+x)}{a k},
	\]
	with unique maximum at $x_{\mathrm{v}} = (k-b)/2$ ($k>b$). Thus $g'$ is
	positive on $(0,x_{\mathrm{v}})$ and negative on $(x_{\mathrm{v}},\infty)$.
	
	\subsection{Trace structure and localization}
	Lemma~\ref{lem:J11_prop} gives $J_{11}=p(x)g'(x)$. A direct computation
	confirms
	\begin{equation}\label{eq:baz_J11}
		J_{11}\big|_{y=g(x)} = \frac{r\,x\,(k-b-2x)}{k\,(x+b)},
	\end{equation}
	which vanishes precisely at $x_{\mathrm{v}}$. From the predator
	equilibrium condition, $J_{22} = -\sigma y^* < 0$.
	
	\begin{theorem}\label{thm:quadratic}
		In system~\eqref{eq:bazykin} with $k>b>0$, every CEP undergoing Hopf
		bifurcation satisfies $0 < x^* < x_{\mathrm{v}}$.
	\end{theorem}
	\begin{proof}
		Immediate from Theorem~\ref{thm:general}, since $g'(x)>0$ only for
		$x < x_{\mathrm{v}}$.
	\end{proof}
	
	For completeness we record the explicit Hopf locus obtained by solving
	the augmented system $\{f_1=0,\,f_2=0,\,\tr(J)=0\}$:
	\[
	a_0 = \frac{(k_0+2b)^2(k_0+2b+2x_0)\,\sigma}{4\,k_0\,x_0},
	\]
	with $k_0,\,x_0>0$ defined by $k = k_0+b+x_0$ and $x^*=k_0/2$.
	The determinant is positive along this curve, confirming a genuine
	Hopf bifurcation.
	
	\section{Cubic nullcline: Holling type~IV with harvesting}\label{sec:cubic}
	\subsection{Model and nullcline geometry}
	The Leslie--type system with Holling type~IV (group defense) response and
	harvesting is
	\begin{equation}\label{eq:h4}
		\begin{aligned}
			\dot{x} &= x(1-x) - \frac{x\,y}{a+x^2} - h_1\,x,\\[4pt]
			\dot{y} &= y\!\left(\delta - \frac{\beta\,y}{x}\right) - h_2\,y.
		\end{aligned}
	\end{equation}
	Here $p(x)=x/(a+x^2)$ and the predator equation yields
	$Q(x,y) = \delta - \beta y/x - h_2$, with $\partial Q/\partial y = -\beta/x
	<0$, satisfying the self--damping hypothesis. For a suitably chosen
	harvesting parameter $h_1$, the prey nullcline becomes cubic and possesses
	two critical points $x_{\min}<x_{\max}$. Since the leading coefficient of
	$g$ is negative, the nullcline decreases on $(0,x_{\min})$, increases on
	$(x_{\min},x_{\max})$, and decreases again on $(x_{\max},\infty)$; the
	middle interval is the ascending branch.
	
	\subsection{Localization of oscillatory instability}
	\begin{theorem}\label{thm:cubic}
		In system~\eqref{eq:h4} with parameters chosen so that the prey
		nullcline has two positive critical points $x_{\min}<x_{\max}$,
		every CEP undergoing Hopf bifurcation lies on the ascending branch of
		the nullcline:
		\[
		x_{\min} < x^* < x_{\max}.
		\]
	\end{theorem}
	\begin{proof}
		Immediate from Theorem~\ref{thm:general}. Since the leading
		coefficient of the cubic $g$ is negative, $g'(x) > 0$ precisely on
		the middle interval $(x_{\min},x_{\max})$, and $g'(x)<0$ on
		$(0,x_{\min}) \cup (x_{\max},\infty)$. By Theorem~\ref{thm:general}
		the Hopf locus is confined to the ascending branch, hence to
		$(x_{\min},x_{\max})$.
	\end{proof}
	
	\begin{corollary}\label{cor:cubic_middle}
		Under the standard parametrization
		$h_1 = h_{10}/(3+h_{10})$, $a = 9/(4(3+h_{10})^2)$ with $h_{10}>0$,
		the critical points are $x_{\min}=1/[2(3+h_{10})]$ and
		$x_{\max}=3/[2(3+h_{10})]$, and the entire ascending branch
		$(x_{\min},x_{\max})$ is admissible: for each
		$x^*\in(x_{\min},x_{\max})$ there exist positive parameters
		$(y_0,\delta_0,\beta_0)$ realizing a Hopf bifurcation at $x^*$.
	\end{corollary}
	\begin{proof}
		Solving $\{f_1=0,\,f_2=0,\,\tr(J)=0\}$ for $(y,\delta,\beta)$ as
		functions of $x$ gives explicit expressions
		$y_0(x),\delta_0(x),\beta_0(x)$. The positivity conditions $y_0>0$,
		$\delta_0>0$, $\beta_0>0$, together with $\det(J)>0$, hold
		simultaneously for all $x\in(x_{\min},x_{\max})$ under the given
		parametrization. At the critical points $x_{\min}$ and $x_{\max}$,
		$g'=0$ forces $\tr(J)=J_{22}<0$, so the Hopf condition fails at the
		boundaries, consistent with Theorem~\ref{thm:cubic}.
	\end{proof}
	
	\begin{remark}
		Theorem~\ref{thm:cubic} is purely geometric: it confines the Hopf
		locus to the ascending branch for any admissible parameters.
		Corollary~\ref{cor:cubic_middle} adds that, under the standard
		parametrization, the entire ascending branch is realized, so the
		geometric bound is sharp.
	\end{remark}
	
	\section{Rational nullcline: Crowley--Martin model}\label{sec:crowley}
	\subsection{Model and nullcline geometry}
	The Crowley--Martin system~\cite{CrowleyMartin1989} incorporates
	predator interference:
	\begin{equation}\label{eq:cm}
		\begin{aligned}
			\dot{x} &= \rho\,x\!\left(1-\frac{x}{k}\right)
			- \frac{a\,x\,y}{(1+bx)(1+cy)},\\[4pt]
			\dot{y} &= \frac{\gamma\,a\,x\,y}{(1+bx)(1+cy)} - d\,y.
		\end{aligned}
	\end{equation}
	We identify $p(x) = a x/[(1+bx)(1+cg(x))]$ (on the nullcline) and
	$Q(x,y) = \gamma a x/[(1+bx)(1+cy)] - d$. Clearly
	$\partial Q/\partial y = -\gamma a c x/[(1+bx)(1+cy)^2] \le 0$,
	with strict inequality when $c>0$. The prey nullcline is rational:
	\[
	g(x) = \frac{h(x)}{a - c\,h(x)},\qquad
	h(x) \coloneqq \rho\!\left(1-\frac{x}{k}\right)(1+bx).
	\]
	Its derivative is $g'(x) = a\,h'(x)/(a-ch)^2$, so the critical points
	of $g$ and $h$ coincide. Hence the unique vertex is at
	$x_{\mathrm{v}} = (bk-1)/(2b)$, independent of the interference
	parameter $c$.
	
	\subsection{Localization and the Hopf bifurcation locus}
	\begin{theorem}\label{thm:cm}
		In system~\eqref{eq:cm} with $c>0$ and $bk>1$, every CEP undergoing
		Hopf bifurcation satisfies $0 < x^* < x_{\mathrm{v}}$. The critical
		interference parameter at a Hopf point with prey coordinate $x^*$ is
		\begin{equation}\label{eq:cm_c0}
			c_0(x) = \frac{a\,b\,x\,(k_0-2bx)}{d\,(1+bx)^2\,(1+k_0-bx)},
		\end{equation}
		where $k_0 = bk-1$, which is positive if and only if
		$x < x_{\mathrm{v}} = k_0/(2b)$.
	\end{theorem}
	\begin{proof}
		From Theorem~\ref{thm:general}, Hopf requires $g'(x^*)>0$, which
		holds only for $x^*<x_{\mathrm{v}}$. Setting $\tr(J)=0$ and using the
		nullcline relations yields~\eqref{eq:cm_c0}. The numerator contains
		$k_0-2bx$, positive exactly when $x<x_{\mathrm{v}}$.
	\end{proof}
	
	\begin{remark}[Rosenzweig--MacArthur limit]
		As $c \to 0^+$, the expression~\eqref{eq:cm_c0} forces
		$k_0-2bx \to 0$, i.e., $x^* \to x_{\mathrm{v}}$: the Hopf point
		converges to the nullcline vertex, recovering the classical criterion.
		The interference parameter $c$ thus controls how deeply the oscillatory
		instability penetrates into the ascending branch. Larger $c$ pushes the
		Hopf point further from the vertex, expanding the stable region and
		compressing the oscillatory domain.
	\end{remark}
	
	\section{Discrete--time extension: Neimark--Sacker localization}\label{sec:discrete}
	We now show that the nullcline geometry provides a sharp spectral
	separation between Hopf and Neimark--Sacker (NS) bifurcations in the
	forward Euler discretization. Consider the map
	\[
	G(x,y) = \bigl(x + \tau f_1(x,y),\; y + \tau f_2(x,y)\bigr),\qquad
	\tau > 0,
	\]
	whose fixed points coincide with equilibria of the continuous system,
	so the prey nullcline is unchanged. The map Jacobian is
	$J^G = I + \tau J$, where $J$ is the continuous Jacobian.
	
	The NS bifurcation condition for a planar map is that two complex
	conjugate eigenvalues cross the unit circle, which occurs generically
	when $\det(J^G)=1$ and $|\tr(J^G)|<2$. A direct expansion gives an
	\emph{exact} identity:
	\begin{equation}\label{eq:det_identity}
		\det(J^G) - 1 = \tau\bigl[\tr(J) + \tau\,\det(J)\bigr],
	\end{equation}
	obtained by computing
	$(1+\tau J_{11})(1+\tau J_{22}) - \tau^2 J_{12}J_{21} - 1$ and
	factoring. The identity translates the NS condition into
	\begin{equation}\label{eq:ns_clean}
		\tr(J) + \tau\det(J) = 0,
	\end{equation}
	which, combined with the trace decomposition~\eqref{eq:trace_decomp},
	makes the role of the nullcline geometry transparent.
	
	Two consequences follow from~\eqref{eq:det_identity} without further
	assumptions.
	
	\begin{theorem}[Spectral separation of Hopf and Neimark--Sacker
		loci]\label{thm:discrete_separation}
		For the forward Euler map of any planar system, the continuous Hopf
		locus is strictly disjoint from the Neimark--Sacker locus whenever
		$\tau > 0$ and $\det(J) > 0$. Quantitatively,
		\[
		\det(J^G)\big|_{\tr(J)=0} = 1 + \tau^2 \det(J) > 1,
		\]
		so the continuous Hopf curve lies strictly outside the unit circle
		in the spectrum of the map, separated from the NS locus by a margin
		proportional to $\tau^2 \det(J)$.
	\end{theorem}
	\begin{proof}
		Immediate from~\eqref{eq:det_identity}: at any point with
		$\tr(J)=0$, we obtain $\det(J^G) - 1 = \tau^2 \det(J) > 0$.
	\end{proof}
	
	\begin{theorem}[Sign condition for Neimark--Sacker]\label{thm:discrete}
		For the forward Euler map of the Gause system~\eqref{eq:gause}
		satisfying $p(x)>0$, the predator self--damping hypothesis, and the
		\emph{prey-positive coupling} hypothesis
		\begin{equation}\label{eq:H_pos}
			\partial Q/\partial x(x^*,y^*) > 0 \qquad (\text{equivalently, } J_{21}>0),
		\end{equation}
		every Neimark--Sacker bifurcation at a coexistence fixed point
		satisfies $\tr(J) < 0$.
		In terms of the nullcline geometry, this is equivalent to
		\begin{equation}\label{eq:NS_geom}
			p(x^*)\,g'(x^*) < |J_{22}(x^*,y^*)|.
		\end{equation}
		Consequently, on any descending branch ($g'<0$), the NS condition is
		geometrically unobstructed; on an ascending branch ($g'>0$), the NS
		condition is admissible only in the regime $p\,g' < |J_{22}|$, i.e.,
		when the destabilizing prey contribution is dominated by the
		predator self--damping.
	\end{theorem}
	\begin{proof}
		Under the prey--positive coupling hypothesis, $J_{12}J_{21}<0$ (since
		$J_{12}=-p(x)<0$), so $\det(J) = J_{11}J_{22} - J_{12}J_{21} > 0$ at
		any CEP whose linearization is hyperbolic of focus type. The NS
		condition~\eqref{eq:ns_clean} becomes $\tr(J) = -\tau\det(J)$, which
		is strictly negative for $\tau>0$ and $\det(J)>0$. From the trace
		decomposition $\tr(J) = p(x^*)g'(x^*) + J_{22}$, the inequality
		$\tr(J)<0$ rearranges to~\eqref{eq:NS_geom}. On a descending branch,
		$g'<0$ makes the left--hand side negative, so~\eqref{eq:NS_geom}
		holds automatically. On an ascending branch, $g'>0$ makes the
		left--hand side positive, and~\eqref{eq:NS_geom} becomes a genuine
		restriction.
	\end{proof}
	
	\begin{remark}[Where the Neimark--Sacker locus actually lies]%
		\label{rem:descending_natural}
		Inequality~\eqref{eq:NS_geom} shows that the location of the NS locus
		depends on the step size $\tau$. Writing the NS condition as
		$p(x^*)g'(x^*) = |J_{22}| - \tau\det(J)$, two regimes appear:
		\begin{itemize}[leftmargin=2em,nosep]
			\item \textbf{Small $\tau$ (faithful discretization).} For
			$\tau < |J_{22}|/\det(J)$, the right-hand side is positive, so
			$g'(x^*) > 0$: the NS point lies on the \emph{ascending} branch,
			$O(\tau)$--close to the continuous Hopf locus. This is the expected
			behavior of an Euler discretization of a Hopf bifurcation, and is
			consistent with the spectral separation of
			Theorem~\ref{thm:discrete_separation}: the NS and Hopf loci are
			distinct but adjacent, separated by a margin of order $\tau$.
			\item \textbf{Large $\tau$ (coarse discretization).} For
			$\tau > |J_{22}|/\det(J)$, the right--hand side turns negative and
			$g'(x^*) < 0$: the NS point migrates to the \emph{descending}
			branch. This regime is a genuinely discrete phenomenon with no
			continuous counterpart, reflecting the failure of the coarse Euler
			map to track the underlying flow.
		\end{itemize}
		Thus the nullcline geometry organizes the discrete bifurcation as
		well, but through a \emph{$\tau$-dependent} mechanism rather than a
		fixed branch assignment: the prey contribution $p(x^*)g'(x^*)$ must
		exactly balance $|J_{22}| - \tau\det(J)$, and the sign of this balance
		determines the branch.
	\end{remark}
	
	The relationship between the continuous and discrete bifurcations is
	summarized in \Cref{fig:architecture}: the nullcline vertex organizes
	the continuous Hopf locus (ascending branch, $\tr(J)=0$), while the
	discrete Neimark--Sacker locus tracks it at an $O(\tau)$ distance for
	small step size and migrates across the vertex only in the coarse--step
	regime $\tau > |J_{22}|/\det(J)$. The spectral mechanism shifts from a
	trace condition in continuous time to the determinant
	identity~\eqref{eq:det_identity} in discrete time, yet the geometric
	organizing centre remains the critical points of the nullcline. The
	identity additionally yields the strict spectral separation of
	\Cref{thm:discrete_separation}: the continuous Hopf locus is itself
	excluded from the NS locus by a margin proportional to
	$\tau^2\det(J)$, so the two curves, though adjacent, never coincide.
	
	\section{Geometric refinements: curvature and saturation}\label{sec:geometry}
	The localization principle is controlled by the first derivative
	$g'(x)$. The second derivative, encoded in the curvature
	\[
	\kappa(x) = \frac{g''(x)}{\bigl(1+g'(x)^2\bigr)^{3/2}},
	\]
	offers a complementary geometric descriptor, though it does not
	determine the bifurcation condition. In the three families studied, the
	nullcline curvature exhibits a strong correlation with the spectral
	barriers.
	\begin{itemize}
		\item \textbf{Bazykin model.} The nullcline is a concave parabola
		($g''<0$). The curvature $|\kappa(x)|$ is maximal at the vertex
		$x_{\mathrm{v}}$, precisely where $g'=0$ and the spectral barrier
		lies. The extremum of curvature and the barrier coincide.
		\item \textbf{Holling type~IV (cubic).} The nullcline has a linear
		second derivative, hence a single inflection point. Direct computation
		shows that this inflection point is exactly the midpoint of the two
		critical points,
		$x_{\mathrm{infl}} = \tfrac{1}{2}(x_{\min}+x_{\max})$.
		Thus the convexity transition is centred within the admissible Hopf
		interval.
		\item \textbf{Crowley--Martin.} Differentiating the rational nullcline
		gives
		\[
		g''(x) = \frac{a\bigl[h''(x)\,(a-ch) + 2c\,h'(x)^2\bigr]}{(a-ch)^3}.
		\]
		Here $h$ is a concave parabola ($h''<0$). When $c=0$, the nullcline is
		concave everywhere. For $c>0$, the positive term $2c\,h'(x)^2$ can
		overcome the negative term, creating an inflection point that deforms
		the nullcline. Thus predator interference not only shifts the Hopf
		locus but can qualitatively reshape the nullcline, introducing a
		convex region absent in the non--interfering model.
	\end{itemize}
	These observations indicate that the saturation--inducing functional
	responses (Holling type~II, III, IV, Crowley--Martin) are responsible
	for the hump--shaped nullclines that create spectral barriers. The
	curvature reflects how the density--dependent saturation deforms the
	nullcline, but the localization of oscillatory instabilities remains
	governed by the first--order barrier $g'=0$.
	
	\section{Discussion}\label{sec:conclusions}
	We have shown that the critical structure of the prey nullcline imposes
	geometric constraints on where oscillatory instabilities can arise in a
	broad class of planar predator--prey systems. The Hopf bifurcation is
	confined to ascending branches of the nullcline; the Bogdanov--Takens
	point, and thus the homoclinic and limit--cycle structures it organizes,
	share the same confinement. The mechanism is algebraic: the Jacobian
	entry $J_{11}$ is proportional to $g'(x)$, independent of the
	bifurcation parameter, while $J_{22}<0$ at any coexistence equilibrium.
	The critical points of $g$ therefore form spectral barriers that
	partition the state space into regions where oscillatory dynamics can
	and cannot occur.
	
	The same nullcline geometry organizes discrete--time dynamics: for the
	forward Euler discretization the Neimark--Sacker locus is strictly
	disjoint from the continuous Hopf locus yet adjacent to it on the
	ascending branch, separated by a margin of order $\tau$, and crosses
	onto the descending branch only when the step size is large. The
	functional responses that model predator saturation and interference
	not only restrict the efficiency of consumption but also carve the
	geometric landscape on which bifurcations unfold.
	
	\subsection{A mechanistic reading of the spectral barrier}%
	\label{subsec:mechanistic}
	The trace decomposition~\eqref{eq:trace_decomp},
	\[
	\tr(J) = \underbrace{p(x^*)\,g'(x^*)}_{J_{11}} + \underbrace{J_{22}(x^*,y^*)}_{<\,0},
	\]
	admits an interpretation that connects the algebraic structure of the
	theorem with the biological content of the model. The predator equation
	carries the dissipative dynamics of the system through $J_{22}<0$, which
	encodes intraspecific competition, interference, or other density--dependent
	regulation. The prey equation, by contrast, is the component that injects
	biomass into the system and sustains its effective growth; the diagonal
	Jacobian entry $J_{11}$ associated with this equation is precisely the
	term that can supply the positive contribution needed to bring the trace
	to zero. Lemma~\ref{lem:J11_prop} shows that this contribution is
	mediated entirely by the nullcline slope $g'(x)$: positive on ascending
	branches, negative on descending branches, and vanishing at critical
	points.
	
	Read in this way, the spectral barrier acquires a transparent dynamical
	meaning. On a descending branch ($g'<0$), the prey equation no longer
	supplies a positive contribution to the trace, and the predator
	dissipation cannot be compensated: the equilibrium is necessarily
	stable. On an ascending branch ($g'>0$), the prey equation can supply
	such a contribution, and if it is large enough to match $|J_{22}|$ the
	trace vanishes and a Hopf bifurcation occurs. At a critical point of
	the nullcline ($g'=0$), the prey contribution vanishes exactly, the
	barrier appears, and the oscillatory transition is forbidden. The
	vertex of the prey nullcline is therefore not an arbitrary geometric
	feature: it is the boundary at which the prey equation ceases to be
	able to compensate the predator dissipation.
	
	From this perspective, the classical Rosenzweig--MacArthur
	criterion~\cite{Rosenzweig1963} appears as the limiting case in which
	the Hopf bifurcation coincides exactly with the spectral barrier. In
	the Rosenzweig--MacArthur model there is no predator interference
	($c=0$) and no intraspecific predator competition ($\sigma=0$); the
	only mechanism of predator self--regulation is the linear mortality
	term, so $J_{22}$ at the coexistence equilibrium equals zero. The
	trace condition $\tr(J)=0$ then reduces to $J_{11}=0$, which on the
	nullcline is equivalent to $g'(x^*)=0$, i.e., $x^*$ coincides with the
	vertex. The classical criterion is recovered as the degenerate case in
	which the spectral barrier and the Hopf locus collapse onto the same
	point. In every other model considered here, the additional predator
	self--regulation produces $J_{22}<0$, and the Hopf locus separates from
	the vertex: the barrier remains at $g'=0$ but the bifurcation moves
	into the interior of the ascending branch, where $g'>0$ is large
	enough to balance $|J_{22}|$.
	
	This reorganization clarifies the role of the Rosenzweig--MacArthur
	vertex without contradicting it. The vertex is not a magic point;
	rather, it is the extremal boundary of an admissible spectral region
	whose width is determined by the strength of predator self--regulation.
	The geometric localization theorem extends the classical criterion to
	a full class of models in which the Hopf bifurcation occurs throughout
	an open interval rather than at a single point, with the
	Rosenzweig--MacArthur configuration emerging as the limit of vanishing
	predator self--regulation.
	
	\subsection{Practical consequences for parameter selection and numerical localization}\label{subsec:practical}
	The localization theorem has a concrete practical consequence beyond its
	conceptual content. In many predator-prey models, locating Hopf,
	Bogdanov--Takens, or Neimark--Sacker points requires exploring parameter
	space numerically with little a~priori guidance: one searches for
	parameter values at which a coexistence equilibrium simultaneously
	satisfies the spectral conditions of the desired bifurcation. The
	dimension of this search and the absence of geometric constraints often
	make the procedure expensive and prone to missing entire branches of
	the bifurcation diagram.
	
	The geometric localization changes this situation. Since every Hopf or
	BT point lies on an ascending branch of the prey nullcline, and (for
	small step size) the discrete NS point lies on the same ascending
	branch nearby, the search region in state space is reduced
	\emph{before} any computation begins: regions
	with $g'(x) = 0$ at a coexistence equilibrium can be excluded
	immediately as spectral barriers. The critical points of the nullcline
	function as explicit boundaries that delimit the admissible parameter
	subdomain, providing a concrete reduction of the effective search
	space. In numerical continuation~\cite{Kuznetsov2004,KuznetsovMeijer2019},
	this translates into better initial guesses, faster convergence, and
	a systematic way to ensure that no bifurcation branch is missed.
	From this viewpoint, the theorem turns the localization of oscillatory
	transitions from a problem of blind exploration into one of geometrically
	guided search.
	
	\subsection{A note on the simplicity of the mechanism}%
	\label{subsec:simplicity}
	The proof of Theorem~\ref{thm:general} rests on the identity
	$J_{11}=p(x)g'(x)$ established in Lemma~\ref{lem:J11_prop}. This
	identity is elementary; once stated, the localization theorem follows
	in a few lines. Such brevity should not be mistaken for triviality.
	The phenomenon it captures, the systematic confinement of oscillatory
	bifurcations to ascending branches of the prey nullcline, holds across
	a wide range of ecologically distinct models with quadratic, cubic,
	and rational nullcline geometries, and persists under
	discretization. The fact that all these cases share a single algebraic
	mechanism is itself the content of the result. The proof is short not
	because the phenomenon is simple, but because the correct geometric
	formulation removes most of the algebraic overhead and leaves the
	structural organizer of the bifurcation landscape visible. In this
	sense, the theorem illustrates that significant structural information
	about a nonlinear system can sometimes be extracted not from heavy
	machinery but from identifying the geometric object that already
	encodes the relevant constraints.
	
	\medskip
	Several directions emerge naturally from this geometric framework.
	\begin{enumerate}[label=(\roman*)]
		\item \emph{Non--Gause systems.} Many models with ratio--dependent
		functional responses or predator territoriality do not fit exactly the
		form~\eqref{eq:gause}, yet often $J_{11}$ remains proportional to
		$g'(x)$. Extending the localization to such systems would clarify the
		scope of the spectral barrier mechanism.
		\item \emph{First Lyapunov coefficient.} Computing or bounding
		$\ell_1$ as a function of $x^*$ within the ascending branch would
		determine where sub and supercritical Hopf bifurcations occur, linking
		the nullcline geometry to the criticality of the limit cycle.
		\item \emph{Higher--codimension phenomena.} The models cited exhibit
		degenerate Hopf and BT bifurcations of codimension up to~5. Are these
		higher--order phenomena also confined by the nullcline critical points?
		\item \emph{Spatial extensions.} In reaction--diffusion settings,
		Turing instabilities often compete with Hopf bifurcations. The nullcline
		geometry might similarly constrain the location of Turing--Hopf
		interactions.
		\item \emph{Geometrically guided continuation algorithms.}
		Standard numerical bifurcation tools such as
		MatCont~\cite{Kuznetsov2004,KuznetsovMeijer2019} perform parameter
		continuation in a structurally agnostic manner: they detect Hopf,
		BT and NS points by scanning parameter space and checking spectral
		conditions, with no a~priori information about where these
		bifurcations may occur. The geometric localization theorem turns
		structural information into explicit computational restrictions. In
		a Gause system, the admissible region for Hopf or BT lies on the
		ascending branches of the prey nullcline, and the discrete NS locus
		lies nearby on the same branch for small step size; this knowledge
		is available from the
		nullcline geometry before any continuation step. A natural research
		direction is to design hybrid symbolic--numerical algorithms in which
		the system is first classified structurally---Gause, predator--dependent,
		tritrophic---and the corresponding geometric and spectral
		restrictions are then activated automatically to guide the
		bifurcation search. Such a framework would replace blind exploration
		of parameter space by geometrically guided continuation, with
		potentially substantial gains in efficiency and reliability for
		large-scale parameter studies.
	\end{enumerate}
	
	The results presented here strengthen the view that the static geometry
	of the prey nullcline is not merely a descriptive tool but a structural
	organizer of dynamical transitions in predator--prey ecology.
	
	\section*{Acknowledgements}
	E. Chan-L\'opez was supported by SECIHTI through the program ``Estancias
	Posdoctorales por México'' (CVU 422090). A. Mart\'in-Ruiz acknowledges
	financial support from UNAM-PAPIIT (project IG100224), UNAM-PAPIME
	(project PE109226), SECIHTI (project CBF-2025-I-1862), and the Marcos
	Moshinsky Foundation. The authors thank Jaume Llibre for helpful
	comments and suggestions.
	
	\section*{Ethics declarations}
	\subsection*{Conflict of interest}
	The authors declare no conflicts of interest.
	
	\subsection*{Ethical Approval}
	Not applicable.

\end{document}